\documentclass[oneside, a4paper, 10.5pt]{article}
\usepackage{kotex}
\usepackage[colorlinks = true,
            linkcolor = black,
            urlcolor  = blue,
            citecolor = black]{hyperref}
\usepackage{amsmath, amsthm, amssymb, amsfonts, amscd}
\usepackage[left=2.5cm,right=2.5cm,top=3cm,bottom=2.5cm,a4paper]{geometry}
\usepackage{mathtools}
\usepackage{thmtools}
\usepackage{scalerel}
\usepackage{graphicx}
\usepackage[titletoc]{appendix}
\usepackage{cleveref}
\usepackage{mathrsfs}
\DeclareGraphicsExtensions{.pdf,.png,.jpg}

\newcommand{\inv}{{\operatorname{inv}}}
\newcommand{\Gal}{\operatorname{Gal}}
\newcommand{\Hom}{\operatorname{Hom}}
\newcommand{\Ext}{\operatorname{Ext}}

\newcommand{\Sh}{\operatorname{Sh}}

\newcommand{\un}{\operatorname{un}}

\newcommand{\Z}{\mathbb{Z}}
\newcommand{\Q}{\mathbb{Q}}
\newcommand{\R}{\mathbb{R}}
\newcommand{\Spec}{\operatorname{Spec}}
\newtheorem{theorem}{Theorem}[section]

\newtheorem{lemma}[theorem]{Lemma}

\newtheorem{remark}[theorem]{Remark}
\newtheorem{definition}[theorem]{Definition}

\newtheorem*{theorem*}{Theorem}
\usepackage{titling}
\usepackage{url}
\usepackage[nosort]{cite}
\newcommand{\changeurlcolor}[1]{\hypersetup{urlcolor=#1}} 

\usepackage{enumitem}
\usepackage{tikz-cd}

\usepackage{array}
\makeatletter
\newcommand{\thickhline}{%
    \noalign {\ifnum 0=`}\fi \hrule height 1pt
    \futurelet \reserved@a \@xhline
}
\newcolumntype{"}{@{\hskip\tabcolsep\vrule width 1pt\hskip\tabcolsep}}
\makeatother

\title{\large{\textbf{ARITHMETIC CHERN-SIMONS THEORY FOR ARITHMETIC SCHEMES}}}
\author{\normalsize{JUNGIN LEE}}
\date{}

\usepackage{fancyhdr}

\newcommand\shorttitle{ARITHMETIC CHERN-SIMONS THEORY FOR ARITHMETIC SCHEMES}
\newcommand\authors{JUNGIN LEE}

\usepackage{sectsty}
\sectionfont{\large}
\subsectionfont{\normalsize}

\begin{document}
\maketitle

\vspace{-10mm}

\begin{abstract}
In this paper, we generalize the arithmetic Chern-Simons theory to regular flat separated schemes of finite type over rings of integers of number fields by applying the duality theorems for arithmetic schemes. 
\end{abstract}

\vspace{2mm}



\section{Introduction} \label{Sec1}

In \cite{ACST1}, Minhyong Kim introduced the arithmetic Chern-Simons theory by applying the Dijkgraaf-Witten theory \cite{DW} for 3-manifolds to the spectrum of rings of integers of totally imaginary number fields based on the analogy of arithmetic topology. 
In a subsequent paper \cite{ACST2} by H.-J. Chung, D. Kim, M. Kim, J. Park and H. Yoo, some explicit computations of the arithmetic Chern-Simons invariant were provided via the decomposition formula.

More general non-vanishing examples of the arithmetic Chern-Simons invariant were studied by F. Bleher, T. Chinburg, R. Greenberg, M. Kakde, G. Pappas and M. Taylor in \cite{BCG}. 
In the author's joint work with J. Park \cite{LP}, the construction and explicit computations of the arithmetic Chern-Simons action in \cite{ACST2} were generalized to arbitrary number fields by the use of cohomology with compact support. Also the non-triviality result of \cite{BCG} was generalized to the large family of non-abelian gauge groups by a simple twisting argument.

Recently, T. Geisser and A. Schmidt \cite{GS} generalized Poitou-Tate duality, Poitou-Tate exact sequence and Artin-Verdier duality to regular, flat, separated schemes of finite type over rings of integers of global fields, which were previously known only for smooth schemes. These results are based on the duality theorems via Bloch's cycle complex in \cite{GEI}. Since the constructions in previous papers rely on Poitou-Tate exact sequence and Artin-Verdier duality, it is natural to ask whether the arithmetic Chern-Simons theory can be generalized to arithmetic schemes.

We give an outline of the paper. 
In Section \ref{Sec2} we set up the basic notations and provide duality theorems that we will use in the remainder of this paper. 
We define arithmetic Chern-Simons action for arithmetic schemes without boundary in Section \ref{Sub32}, based on the cohomological results in Section \ref{Sub31}. 
Section \ref{Sub33} and \ref{Sub34} are devoted to the following decomposition formula and its proof. 
\begin{theorem*}
(The decomposition formula) Let $F$ be a number field, $n \geq 2$ be an integer and $X$ be a regular, flat, separated scheme of finite type of relative dimension $r$ over $B = \Spec \mathcal{O}_F$ such that $\mu_{n, X}(X) \cong \Z/n\Z$. For any finite set $T_f$ of finite primes of $F$ containing all primes dividing $n$,
$$
CS(\rho) = \sum_{v \in T_f} \inv_v([b_{-,v}-b_{+,v}])
$$
for $\rho \in H^{2r+3}(\pi^{\un}, \mathscr{M})$. (See Section \ref{Sec2} and \ref{Sec3} for notations.)
\end{theorem*}
Even though we do not provide explicit computations of the arithmetic Chern-Simons action in this paper, we expect that this will lead to interesting results for arithmetic schemes.


\section{Preliminaries} \label{Sec2}


\subsection{Definitions and notations} \label{Sub21}

\noindent Let $F$ be a number field and $n \geq 2$ be an integer. Let $S_{\infty}$ be the set of all real places of $F$, $S_f$ be any finite set of finite places of $F$ and $S = S_f \cup S_{\infty}$. 
In particular, denote $T_{0f}$ be the set of places of $F$ above $n$, $T_f$ be a finite set of places of $F$ containing $T_{0f}$, $T_0=T_{0f} \cup S_{\infty}$ and $T=T_f \cup S_{\infty}$.
Note that since there are no ramifications at complex places of $F$, it does not affect the theory if we include complex places in $S_{\infty}$ as in \cite{CM} or \cite{GS}.

Denote the base scheme $\Spec \mathcal{O}_F$ by $B$ and let $B_S := \Spec(\mathcal{O}_F \left [ \frac{1}{S_f} \right ])$. Let $X$ be a regular, flat, separated scheme of finite type of relative dimension $r$ over $B$ with a structure morphism $f : X \rightarrow B$, $X_S :=  X \times_B B_S$ and $f_S$, $g_S$ and $g'_S$ are given as below. Note that $g_S$ and $g_S'$ are open immersions and $f_S$ is a regular, flat, separated morphism of finite type of relative dimension $r$.
\[
\begin{tikzcd}
X_S \arrow[r, "f_S"] \arrow[d, "g'_S"] & B_S \arrow[d, "g_S"]  \\
X \arrow[r, "f"] & B
\end{tikzcd}
\]

Let $\text{G}_F := \Gal(\overline{F}/F)$, $\mathfrak{b} : \Spec \overline{F} \rightarrow B$ be a geometric point and $\pi := \pi_1(B, \mathfrak{b}) \cong \Gal(F^{f}_{\un}/F)$ where $F^f_{\un}$ is the maximal extension of $F$ in $\overline{F}$ unramified at all finite places. Denote $\pi^{\un} := \Gal(F_{\un}/F)$ where $F_{\un}$ is the maximal extension of $F$ in $\overline{F}$ unramified at all places. 
Let $\pi_v := \Gal(\overline{F_v}/F_v)$, $I_v \subset \pi_v$ be the inertia group and $\kappa_v : \pi_v \rightarrow \text{G}_F \rightarrow \pi^{\un}$ be given by choices of embeddings $\overline{F} \rightarrow \overline{F_v}$.

Let $\widetilde{S}$ be any set of places of $F$ satisfying $S_f \subseteq \widetilde{S} \subseteq S$ and $\pi_{\widetilde{S}} := \Gal(F_{\un}^{\widetilde{S}}/F)$ where $F_{\un}^{\widetilde{S}}$ is the maximal unramified extension of $F$ in $\overline{F}$ unramified outside $\widetilde{S}$. 
Then $\pi_S \cong \pi_1(B_S, \mathfrak{b}_S)$ for a geometric point $\mathfrak{b}_S : \Spec \overline{F} \rightarrow B_S$. For $\varpi := \Gal(F_{\un}^f/F_{\un})$ and $\varpi_S := \Gal(F_{\un}^{S}/F_{\un}^{S_f})$, $\pi / \varpi \cong \pi^{\un}$ and $\pi_S/\varpi_{S} \cong \pi_{S_f}$. 
Let $p_S : \pi_S \rightarrow \pi_{S_f}$ and $\kappa_{\widetilde{S}} : \pi_{\widetilde{S}} \rightarrow \pi^{un}$ be natural quotient maps.

For any place $v$ of $F$, denote $B_v := \Spec F_v $ and $X_v := X \times_B B_v = X_S \times_{B_S} B_v$. For an \'etale sheaf $\mathscr{F}$ on $X_S$, its localization on $X_v$ is denoted by $\mathscr{F}_v$. 
For each $v \in \widetilde{S}$, let $i_v = i_{v, \widetilde{S}} : \pi_v \rightarrow \text{G}_F \rightarrow \pi_{\widetilde{S}}$ be given by the embedding $\overline{F} \rightarrow \overline{F_v}$ same as above and $i'_v = i'_{v, S} : B_v \rightarrow B_S$ be the natural map. 
For a $\pi_{\widetilde{S}}$-module $M$, let $M_v$ be $M$ equipped with a $\pi_v$-module structure given by $i_v$.

For an abelian category $\mathscr{A}$, denote the derived category of bounded complexes over $\mathscr{A}$ by $\mathscr{D}^b(\mathscr{A})$. 
Unless otherwise stated, every sheaf or cohomology group is \'etale.
Assume that $\mu_{n, X}(X) \cong \Z/n\Z$, where $\mu_{n, X}$ is the sheaf of $n$-th roots of unity of $X$. For $X=B$, this is equivalent to the condition that $\mu_n(\overline{F}) = \mu_n(F)$. 
Since $g_S'$ is an open immersion, $\mu_{n, X_S}(X_S)=\mu_{n, X}(X) \cong \Z/n\Z$ for any $S$.

A Pontryagin dual of a finite abelian group $G$ is denoted by $G^D$. For a Galois group $G=\Gal(L/K)$, denote the dual group of a $G$-module $M$ by $M^{\vee} := \Hom_G(M, L^{\times})$. For a locally constant abelian \'etale sheaf $\mathscr{F}$ on $X_S$, its dual is denoted by $\mathscr{F}^{\vee} := \underline{\Hom}_{X_S}(\mathscr{F}, \mathbf{G}_m)$. For a sheaf of $\Z/n\Z$-module $\mathscr{F}$ on $X_S$, define $\mathscr{F}^{\vee}(r) := \underline{\Hom}_{X_S}(\mathscr{F}, \mu_{n, X_S}^{\otimes r})$.


\subsection{Duality for arithmetic schemes} \label{Sub22}

In this subsection, we introduce Poitou-Tate exact sequence and local and global duality theorems for arithmetic schemes following the exposition given in \cite{GS}. First we fix some notations about modified and compactly supported \'etale cohomology.

Let $A$ be a ring which is one of $\R$, $\Z$ and a non-archimedean local field and $\mathscr{F}^{\bullet}$ be a bounded complex of sheaves on $\Spec A$. For each ring $A$, define $\widehat{H}^i(\Spec A, \mathscr{F}^{\bullet})$ ($i \in \Z$) as in \cite[p. 2023-2024]{GS}.

\begin{definition} \label{def221}
Let $A$ be a ring as above, $h : \mathscr{X} \rightarrow \Spec A$ be a separated morphism of finite type, $\mathscr{F}^{\bullet}$ be a bounded complex of torsion sheaves on $\mathscr{X}$ and $i \in \Z$. \\
(1) $\widehat{H}^i(\mathscr{X}, \mathscr{F}^{\bullet}) := \widehat{H}^i(\Spec A, Rh_* \mathscr{F}^{\bullet})$ \textup{(modified \'etale cohomology).} \\
(2) $H_c^i(\mathscr{X}, \mathscr{F}^{\bullet}) := H^i(\Spec A, Rh_! \mathscr{F}^{\bullet})$ \textup{(\'etale cohomology with compact support).} \\
(3) $\widehat{H}^i_c(\mathscr{X}, \mathscr{F}^{\bullet}) := \widehat{H}^i(\Spec A, Rh_! \mathscr{F}^{\bullet})$ \textup{(modified \'etale cohomology with compact support).}
\end{definition}

\begin{remark} \label{rmk222}
If $\mathscr{X} = B_S$ and $\mathscr{F}$ is a single sheaf on $B_S$, then the groups $\widehat{H}^i(B_S, \mathscr{F})$ ($i \in \Z$) coincide with the modified \'etale cohomology groups defined in \cite[Definition 3.1.4]{ZINK} and the groups $\widehat{H}_c^i(B_S, \mathscr{F})$ ($i \in \Z$) coincide with the compactly supported cohomology groups defined in \cite[p. 166]{ADT} or \cite[Definition 5.4.1]{CM}. 
See \cite[Section A.3]{LP} for the identification of two definitions of compactly supported cohomology groups. 
\end{remark}

\begin{remark} \label{rmk223}
One should be careful about the notations about cohomology groups. As mentioned above, the groups $\widehat{H}_c^i(B_S, \mathscr{F})$ are same as the compactly supported cohomology groups, which is denoted by $H_c^i(B_S, \mathscr{F})$ in \cite{CM}, \cite{LP} and \cite{ADT}. We use the notation $\widehat{H}_c^i$ here. 
For Galois cohomology groups with compact support, we use the notation $H_c^i$ as usual. 
\end{remark}

There are two versions of Poitou-Tate exact sequence for arithmetic schemes: the original one \cite[Theorem B]{GS} and its dual version with compact support \cite[Theorem C]{GS}. We will only introduce the second version because we will not use the first one.

\begin{theorem} \label{thm223}
\textup{(\cite[Theorem C]{GS}; Poitou-Tate exact sequence with compact support)} Let $\mathscr{F}$ be a locally constant, constructible sheaf of $\Z/n\Z$-modules on $X_T$. \\
(1) There exists a $(6r+9)$-term exact sequence of abelian topological groups and strict homomorphisms
\begin{equation*}
\begin{split}
0 \longrightarrow & H^{0}_c(X_T, \mathscr{F}) \longrightarrow P^{0}_c(X_T, \mathscr{F}) \longrightarrow H^{2r+2}(X_T, \mathscr{F}^{\vee} (r+1))^D \longrightarrow \\
\cdots \longrightarrow & H^{i}_c(X_T, \mathscr{F}) \xrightarrow{\lambda_{i,c}} P^{i}_c(X_T, \mathscr{F}) \longrightarrow H^{2r+2-i}(X_T, \mathscr{F}^{\vee} (r+1))^D \longrightarrow \\
\cdots \longrightarrow & H^{2r+2}_c(X_T, \mathscr{F}) \longrightarrow P^{2r+2}_c(X_T, \mathscr{F}) \longrightarrow H^{0}(X_T, \mathscr{F}^{\vee} (r+1))^D \longrightarrow 0
\end{split}
\end{equation*}
where $\displaystyle P^{i}_c(X_T, \mathscr{F}) := \prod_{v \in T} \widehat{H}^i_c(X_v, \mathscr{F}_v)$. \\
(2) For $i \geq 2r+3$, the localization map
$$
\lambda_{i,c} : H^i_c(X_T, \mathscr{F}) \rightarrow P^i_c(X_T, \mathscr{F}) = \prod_{v \in X_{\infty}} \widehat{H}^i_c(X_v, \mathscr{F}_v)
$$
is an isomorphism. 
\end{theorem}

Now we provide local and global duality theorems for arithmetic schemes. The theorems are stated in a slightly different manner from \cite{GS}.

\begin{theorem} \label{thm224}
\textup{(Local duality)} Let $v \in T$ and $\mathscr{F}$ be a locally constant, constructible sheaf of $\Z/n\Z$-modules on $X_v$. Then Tate's local duality induces perfect pairing of finite abelian groups
$$
\widehat{H}_c^i (X_v, \mathscr{F}) \times H^{(2r+2)-i}(X_v, \mathscr{F}^{\vee} (r+1))  \rightarrow \Q/\Z
$$
for all $i \in \Z$. 
\end{theorem}

By Theorem \ref{thm224}, for $v \in T$,
$$
\widehat{H}_c^{2r+2}(X_v, \Z/n\Z) \cong H^0(X_v, \Z/n\Z^{\vee}(r+1))^D
$$
$$
= \Hom (\Z/n\Z, \mu_{n, X_v}(X_v)^{\otimes (r+1)})^D = (\mu_{n, X_v}(X_v)^D)^{\otimes (r+1)} \cong \frac{1}{n}\Z/\Z.
$$
Denote the isomorphism $\widehat{H}_c^{2r+2}(X_v, \Z/n\Z) \xrightarrow{\simeq} \frac{1}{n}\Z/\Z$ by $\inv_v$.

\begin{proof}
By Step 2 in \cite[p.2039]{GS} and \cite[Proposition 10.1]{GS}, for all $i \in \Z$, 
$$
H^{(2r+2)-i}(X_v, \mathscr{F}^{\vee} (r+1)) \cong \Ext^{3-i}_{X_v}(\mathscr{F}, \Z^c_{X_v}(-1)).
$$
(Step 2 provides isomorphisms for a scheme $X_{(v)} := X \times_B \Spec F_{(v)}$ over a henselization $F_{(v)}$ of $F$ at $v$ and \cite[Proposition 10.1]{GS} enables us to replace henselization by completion.) Now the theorem follows from \cite[Theorem 4.3]{GS}. 
\end{proof}

\begin{theorem} \label{thm225}
\textup{(Generalized Artin-Verdier duality)} Let $\mathscr{F}$ be a locally constant, constructible sheaf of $\Z/n\Z$-modules on $X_T$. Then Artin-Verdier duality induces perfect pairing of finite abelian groups
$$
\widehat{H}_c^i (X_T, \mathscr{F}) \times H^{(2r+3)-i}(X_T, \mathscr{F}^{\vee} (r+1))  \rightarrow \Q/\Z
$$
for all $i \in \Z$. 
\end{theorem}

\begin{proof}
By Step 1 in \cite[p.2039]{GS}, for all $i \in \Z$, 
$$ 
H^{(2r+3)-i}(X_T, \mathscr{F}^{\vee} (r+1)) \cong \Ext^{2-i}_{X_T}(\mathscr{F}, \Z^c_{X_T}(0)). 
$$
(Note that there is a typo in \cite[p.2039, Step 1]{GS}: each of $\Z^c_{\mathscr{X}}(0)$ should be changed into $\Z^c_{\mathscr{X_{\mathscr{S}}}}(0)$.)
Now the theorem follows from \cite[Theorem 4.6]{GS} since $g_T \circ f_T = f \circ g_T' : X_T \rightarrow B$ is a separated morphism of finite type.  
\end{proof}

\begin{remark} \label{rmk226}
The theorem above is a generalization of \cite[Corollary 2.7.7]{ADT} to the non-smooth case. 
\end{remark}

By Theorem \ref{thm225}, 
$$
\widehat{H}^{2r+3}_c(X_T, \Z/n\Z) \cong H^0(X_T, \Z/n\Z^{\vee}(r+1))^D
$$
$$
= \Hom (\Z/n\Z, \mu_{n, X_T}(X_T)^{\otimes (r+1)})^D = (\mu_{n, X_T}(X_T)^D)^{\otimes (r+1)} \cong \Z/n\Z.
$$
Denote the isomorphism $\widehat{H}^{2r+3}_c(X_T, \Z/n\Z) \xrightarrow{\simeq} (\mu_{n, X_T}(X_T)^D)^{\otimes (r+1)}$ by $\inv_T$.


\section{Arithmetic Chern-Simons action for arithmetic schemes} \label{Sec3}

In this section, we define the arithmetic Chern-Simons action for arithmetic schemes and prove the decomposition formula. 


\subsection{Galois and \'etale cohomology} \label{Sub31}

Let $\mathbf{Et}(B_S)$ be the small \'etale site of $B_S$ and $\mathbf{FSet}_{\pi_S}$ be the category of finite continuous $\pi_S$-sets with a natural Grothendieck topology. 
Let $\mathbf{FEt}(B_S)$ be a full subcategory of the category of $B_S$-schemes whose objects are finite \'etale morphisms to $B_S$, with a natural Grothendieck topology. 
Then there is an equivalence of categories $\mathbf{FEt}(B_S) \rightarrow \mathbf{FSet}_{\pi_S}$ and there is a natural morphism of sites $f : \mathbf{Et}(B_S) \rightarrow \mathbf{FEt}(B_S)$ given by a functor $\mathbf{FEt}(B_S) \rightarrow \mathbf{Et}(B_S)$. Also $f$ induces a morphism of complexes of sheaves
$$
\alpha_* : \mathscr{D}^b(\Sh(\mathbf{Et}(B_S))) \rightarrow \mathscr{D}^b(\Sh(\mathbf{FEt}(B_S))).
$$
For $\mathscr{F}^{\bullet} \in \mathscr{D}^b(\Sh(\mathbf{Et}(B_S)))$, denote the corresponding element of $\alpha_* \mathscr{F}^{\bullet} \in \mathscr{D}^b(\Sh(\mathbf{FEt}(B_S)))$ in $\mathscr{D}^b (\mathbf{Mod}_{\pi_S})$ ($\mathbf{Mod}_{\pi_S}$ is the category of $\pi_S$-modules) also by $\alpha_* \mathscr{F}^{\bullet}$. Then there is a natural map
$$
j^i : H^i(\pi_S, \alpha_* \mathscr{F}^{\bullet}) \cong H^i(\mathbf{FEt}(B_S), \alpha_* \mathscr{F}^{\bullet}) \rightarrow H^i(\mathbf{Et}(B_S), \mathscr{F}^{\bullet}) = H^i(B_S, \mathscr{F}^{\bullet})
$$
and
$$
j^i_c : H^i_c(\pi_S, \alpha_* \mathscr{F}^{\bullet}) \rightarrow \widehat{H}^i_c(B_S, \mathscr{F}^{\bullet})
$$
(see Remark \ref{rmk223} for notations). In the next subsection, we will use the map 
$$
j^{2r+3}_c : H^{2r+3}_c(\pi, \alpha_* Rf_! \Z/n \Z) \rightarrow H^{2r+3}_c(B, Rf_! \Z/n \Z) = \widehat{H}^{2r+3}_c(X, \Z/n \Z).
$$
to define the arithmetic Chern-Simons action for arithmetic schemes. \\

We constructed a map connecting Galois and \'etale cohomology above. We also need maps between Galois cohomology groups and between \'etale cohomology groups. Define 
\begin{center}
$\mathscr{M}_{\widetilde{S}} := (\alpha_* Rf_{S!} \Z/n\Z_{X_S})^{\Gal(F_{\text{un}}^S / F_{\text{un}}^{\widetilde{S}})}$, \\ 
$\mathscr{M} := \mathscr{M}_{\phi} = (\alpha_* Rf_! \Z/n\Z_{X})^{\varpi}$, \\
$\mathscr{G}_S := Rf_{S!} \Z/n\Z_{X_S}$ and $\mathscr{G} := \mathscr{G}_{S_{\infty}} = Rf_! \Z/n\Z_{X}$. 
\end{center}

\begin{lemma} \label{lem311}
For $\widetilde{S_1} \subset \widetilde{S_2}$ and an open immersion $j : B_{S_2} \rightarrow B_{S_1}$, 
\begin{center}
$\mathscr{G}_{S_2} \cong j^* \mathscr{G}_{S_1}$ and
$\mathscr{M}_{\widetilde{S_2}}^{\Gal(F_{\text{un}}^{\widetilde{S_2}} / F_{\text{un}}^{\widetilde{S_1}})} \cong \mathscr{M}_{\widetilde{S_1}}$.
\end{center}
\end{lemma}

\begin{proof}
\[
\begin{tikzcd}
X_{S_2} \arrow[r, "f_{S_2}"] \arrow[d, "j'"] & B_{S_2} \arrow[d, "j"]  \\
X_{S_1} \arrow[r, "f_{S_1}"] & B_{S_1}
\end{tikzcd}
\]
Consider the following diagram. By \cite[Theorem 7.4.4(i)]{LF}, 
$$
\mathscr{G}_{S_2} = Rf_{S_2 !} \Z/n\Z_{X_{S_2}} = Rf_{S_2 !} j'^* \Z/n\Z_{X_{S_1}} \cong j^* Rf_{S_1 !}\Z/n\Z_{X_{S_1}} = j^* \mathscr{G}_{S_1}
$$
and
\begin{equation*}
\begin{split}
\mathscr{M}_{\widetilde{S_2}}^{\Gal(F_{\text{un}}^{\widetilde{S_2}} / F_{\text{un}}^{\widetilde{S_1}})} 
& = ((\alpha_* \mathscr{G}_{S_2})^{\Gal(F_{\text{un}}^{S_2} / F_{\text{un}}^{\widetilde{S_2}})})^{\Gal(F_{\text{un}}^{\widetilde{S_2}} / F_{\text{un}}^{\widetilde{S_1}})}  \\
& \cong ((\alpha_* j^* \mathscr{G}_{S_1})^{\Gal(F_{\text{un}}^{S_2} / F_{\text{un}}^{S_1})})^{\Gal(F_{\text{un}}^{S_1} / F_{\text{un}}^{\widetilde{S_1}})}  \\
& \cong (\alpha_* \mathscr{G}_{S_1})^{\Gal(F_{\text{un}}^{S_1} / F_{\text{un}}^{\widetilde{S_1}})}
= \mathscr{M}_{\widetilde{S_1}}.
\end{split}
\end{equation*}
The isomorphism $(\alpha_* j^* \mathscr{G}_{S_1})^{\Gal(F_{\text{un}}^{S_2} / F_{\text{un}}^{S_1})} \cong \alpha_* \mathscr{G}_{S_1}$ is true because $j$ is \'etale so the sheaf $j^* \mathscr{G}_{S_1}$ is defined by $j^* \mathscr{G}_{S_1} (Y) = \mathscr{G}_{S_1} (Y)$ for an \'etale scheme $Y$ over $B_{S_1}$.
\end{proof}

Let $\mathscr{M}_{\widetilde{S}, v} := i^*_{v, \widetilde{S}} \mathscr{M}_{\widetilde{S}}$ be a $\pi_v$-module. For any $\widetilde{S_1} \subset \widetilde{S_2}$ and a projection $\kappa_{\widetilde{S_1}, \widetilde{S_2}} : \pi_{\widetilde{S_2}} \rightarrow \pi_{\widetilde{S_1}}$, there is a canonical injection
$$
\mathscr{M}_{\widetilde{S_1}, v} 
= i^{*}_{v, \widetilde{S_1}}(\mathscr{M}_{\widetilde{S_2}}^{\Gal(F_{\text{un}}^{\widetilde{S_2}} / F_{\text{un}}^{\widetilde{S_1}})} )
= i^{*}_{v, \widetilde{S_2}} \kappa^{*}_{\widetilde{S_1}, \widetilde{S_2}}(\mathscr{M}_{\widetilde{S_2}}^{\Gal(F_{\text{un}}^{\widetilde{S_2}} / F_{\text{un}}^{\widetilde{S_1}})} )
\rightarrow i^{*}_{v, \widetilde{S_2}} \mathscr{M}_{\widetilde{S_2}}
= \mathscr{M}_{\widetilde{S_2}, v} .
$$
Denote the $\pi_v/I_v$-module $(\mathscr{M}_{\widetilde{S}, v})^{I_v}$ by $\mathscr{N}_{\widetilde{S}, v}$. 
Then $H^q(\mathscr{N}_{\widetilde{S}, v})=0$ for any $q>2r$ by \cite[Theorem 7.4.5]{LF} and $H^2(\pi_v/I_v, A)=H^3(\pi_v/I_v, A)=0$ for any $n$-torsion, finite $\pi_v/I_v$-module $A$ by \cite[Proposition 2.18]{GC}. Combining these two facts, we obtain 
$$
H^{2r+2}(\pi_v/I_v, \mathscr{N}_{\widetilde{S}, v})=H^{2r+3}(\pi_v/I_v, \mathscr{N}_{\widetilde{S}, v})=0.
$$
Following the argument of \cite[Section 3 (2)]{LP}, for $\widetilde{S_1} \subset \widetilde{S_2}$, we get a canonical map
$$
\beta^{\Gal}_{\widetilde{S_1}, \widetilde{S_2}} : 
H^{2r+3}_c(\pi_{\widetilde{S_1}}, \mathscr{M}_{\widetilde{S_1}}) \rightarrow
H^{2r+3}_c(\pi_{\widetilde{S_2}}, \mathscr{M}_{\widetilde{S_2}}). 
$$
Similarly, for $S_1 \subset S_2$, there is a canonical map
$$
\beta^{\text{\'et}}_{S_1, S_2} : 
\widehat{H}^{2r+3}_c(X_{S_1}, \Z/n\Z) = \widehat{H}^{2r+3}_c(B_{S_1}, \mathscr{G}_{S_1}) \rightarrow \widehat{H}^{2r+3}_c(B_{S_2}, \mathscr{G}_{S_2}) = \widehat{H}^{2r+3}_c(X_{S_2}, \Z/n\Z).
$$


\subsection{Arithmetic Chern-Simons action without boundary} \label{Sub32}

In the process of generalizing the arithmetic Chern-Simons action without boundary, two problems emerge. First, it is unclear how to give an element of $H^{2r+3}_c(\pi, \alpha_* Rf_! \Z/n \Z)$ by appropriate $\rho$ and $c \in H^{2r+3}(A, \Z/n\Z)$. To avoid this problem, we define the arithmetic Chern-Simons action by a function defined on $H^{2r+3}_c(\pi, \alpha_* Rf_! \Z/n \Z)$. This is a limitation of our paper. Our generalization is not an actual arithmetic Chern-Simons theory for arithmetic schemes. It is just a generalization of cohomological construction of the arithmetic Chern-Simons theory. 

Secondly, generalized Artin-Verdier duality cannot be applied to the scheme $X$. 
This problem can be overcome by composing a map 
$$
\beta^{\text{\'et}}_{S_{\infty}, T_0} : \widehat{H}^{2r+3}_c(X, \Z/n \Z) \rightarrow \widehat{H}^{2r+3}_c(X_{T_0}, \Z/n \Z)
$$ 
with the map 
$$
\inv_{T_0} : \widehat{H}^{2r+3}_c(X_{T_0}, \Z/n \Z) \rightarrow (\mu_{n, X_{T_0}}(X_{T_0})^D)^{\otimes (r+1)}.
$$

Let $\rho \in H^{2r+3}(\pi^{\un}, \mathscr{M})$ and $j^{2r+3}_{\un}$ be the map defined by the composition
$$
j^{2r+3}_{\un} : H^{2r+3}(\pi^{\un}, \mathscr{M}) 
\xrightarrow{\beta^{\Gal}_{\phi, S_{\infty}}} 
H^{2r+3}_c(\pi, \mathscr{M}_{S_{\infty}}) \xrightarrow{j^{2r+3}_c} 
\widehat{H}^{2r+3}_c(X, \Z/n \Z).
$$
Define the \textbf{arithmetic Chern-Simons action} by a function
\begin{equation}
CS : H^{2r+3}(\pi^{\un}, \mathscr{M}) \rightarrow (\mu_{n, X_{T_0}}(X_{T_0})^D)^{\otimes (r+1)} \,\,\, (\rho \mapsto (\inv_{T_0} \circ \beta^{\text{\'et}}_{S_{\infty}, T_0} \circ j^{2r+3}_{\un})(\rho) ).
\end{equation}
By the assumption $\mu_{n, X}(X) \cong \Z/n\Z$, $(\mu_{n, X_{T}}(X_{T})^D)^{\otimes (r+1)} \cong \frac{1}{n}\Z/\Z$ as abelian groups. To give compatible isomorphisms $(\mu_{n, X_{T}}(X_{T})^D)^{\otimes (r+1)} \cong \frac{1}{n}\Z/\Z$ for each $T \supset T_0$, we need the following lemma. 

\begin{lemma}
For any $T \supset T_0$, the map $\beta^{\text{\'et}}_{T_0, T}$ is an isomorphism. 
\end{lemma}

\begin{proof}
By Theorem \ref{thm225}, the following diagram commutes. 
\[
\begin{tikzcd}
\widehat{H}^{2r+3}_c(X_{T_0}, \Z/n \Z) \arrow[d, "\beta^{\text{\'et}}_{T_0, T}"] & \times &
H^0(X_{T_0}, \mathscr{F}^{\vee} (r+1)) \arrow[d, "\gamma"] \arrow[r] & \Q/\Z \arrow[d, equal] \\
\widehat{H}^{2r+3}_c(X_{T}, \Z/n \Z)  & \times &
H^0(X_{T}, \mathscr{F}^{\vee} (r+1)) \arrow[r] & \Q/\Z
\end{tikzcd}
\]
The map
$$
\gamma : (\mu_{n, X_{T_0}}(X_{T_0}))^{\otimes (r+1)} \rightarrow (\mu_{n, X_{T}}(X_{T}))^{\otimes (r+1)}
$$
is induced by an open immersion $X_T \rightarrow X_{T_0}$, so it is an isomorphism. Since each row is a perfect pairing, $\beta^{\text{\'et}}_{T_0, T}$ is also an isomorphism. 
\end{proof}

Fix an isomorphism $(\mu_{n, X_{T}}(X_{T})^D)^{\otimes (r+1)} \cong \frac{1}{n}\Z/\Z$ for each $T \supset T_0$ such that the following diagram commutes. 
\[
\begin{tikzcd}
\widehat{H}^{2r+3}_c(X_{T_0}, \Z/n \Z) \arrow[d, "\inv_{T_0}", "\simeq"'] \arrow[rr, "\beta^{\text{\'et}}_{T_0, T}", "\simeq"'] & & \widehat{H}^{2r+3}_c(X_T, \Z/n \Z) \arrow[d, "\inv_T", "\simeq"'] \\
(\mu_{n, X_{T_0}}(X_{T_0})^D)^{\otimes (r+1)} \arrow[r, "\simeq"] & \frac{1}{n}\Z/\Z  & (\mu_{n, X_T}(X_T)^D)^{\otimes (r+1)} \arrow[l, "\simeq"']
\end{tikzcd}
\]
Now identify $(\mu_{n, X_{T}}(X_{T})^D)^{\otimes (r+1)}$ with $\frac{1}{n}\Z/\Z$ for each $T \supset T_0$ as above and consider the map $CS$ as the map from $H^{2r+3}(\pi^{\un}, \mathscr{M})$ to $\frac{1}{n}\Z/\Z$.


\subsection{Boundary} \label{Sub33}

This subsection serves as a preparation for the proof of the decomposition formula. Actually this subsection corresponds to the Section 2.4 of \cite{LP} about arithmetic Chern-Simons with boundary, but we do not define the arithmetic Chern-Simons action with boundary here.

Let $H_T^i := \prod_{v \in T} H^i(\pi_v, \mathscr{M}_{T, v})$. Then for a base change $f_v : X_v \rightarrow B_v$ $(v \in T)$ of $f$ and a morphism
$$
\alpha_{v*} : \mathscr{D}^b(\Sh(\mathbf{Et}(B_v))) \rightarrow \mathscr{D}^b(\mathbf{Mod}_{\pi_v}),
$$
$$
\mathscr{M}_{T, v}
= i^{*}_{v, T} \alpha_* Rf_{T!} \Z/n\Z
= \alpha_{v *} i{'}^*_{v,T} Rf_{T!} \Z/n\Z
= \alpha_{v *} Rf_{v!} \Z/n\Z_v
$$
so
$$
H_T^{2r+2} = \prod_{v \in T} \widehat{H}^{2r+2}(B_v, Rf_{v !} \Z/n\Z_v) = \prod_{v \in T} \widehat{H}_c^{2r+2}(X_v, \Z/n\Z_v) = P^{2r+2}_c(X_T, \Z/n\Z). 
$$
By Theorem \ref{thm223}, there exists an exact sequence
$$
H^{2r+2}_c(X_T, \Z/n\Z) \rightarrow P^{2r+2}_c(X_T, \Z/n\Z) \rightarrow H^0(X_T, \Z/n\Z^{\vee}(r+1))^D \rightarrow 0.
$$
Denote the map
$$
H_T^{2r+2} = P^{2r+2}_c(X_T, \Z/n\Z) \rightarrow H^0(X_T, \Z/n\Z^{\vee}(r+1))^D = (\mu_{n, X_T}(X_T)^D)^{\otimes (r+1)} \xrightarrow{\simeq} \frac{1}{n}\Z/\Z
$$
by $\sum$. By the same argument as in \cite[Remark 2.6]{LP}, we obtain the relation 
$$
\textstyle{\sum} = \displaystyle \sum_{v \in T} \inv_v
$$
for the maps $\inv_v$ defined in Section \ref{Sub22}.

By \cite[Proposition 8.3.18]{NSW} and \cite[Theorem 10.6.1]{NSW}, $H^3(\pi_{T_f}, A)=0$ for any $n$-torsion $\pi_{T_f}$-module $A$. By \cite[Theorem 7.4.5]{LF},
\begin{center}
$H^q(\mathscr{G}_T) = R^q f_{T !} \Z/n\Z_{X_T} =0$ and 
$H^q((\alpha_* \mathscr{G}_T)^{\varpi_T}) = 0$
\end{center}
for any $q>2r$. Note that 
$$
(\alpha_* \mathscr{G}_T)^{\varpi_T} \in \mathscr{D}^b(\mathbf{Mod}_{\pi_{T_f}}),
$$ 
where $\mathbf{Mod}_{\pi_{T_f}}$ is the category of continuous $\pi_{T_f}$-modules. 
Combining these two facts, we obtain that 
$$
H^{2r+3}(\pi_{T_f}, \mathscr{M}_{T_f}) = H^{2r+3}(\pi_{T_f}, (\alpha_* \mathscr{G}_T)^{\varpi_T}) = 0. 
$$


\subsection{Decomposition formula} \label{Sub34}

In this subsection, we provide a decomposition formula for arithmetic schemes and its proof. We can represent the arithmetic Chern-Simons action by the sum of local invariant as \cite[Theorem 3.1]{LP}. Even though we do not represent the arithmetic Chern-Simons action by the different of two elements of some $\frac{1}{n}\Z/\Z$-torsor, our decomposition formula essentially contains the information of the difference of a local unramified trivialization and a global ramified trivialization. 
The proof is similar to the proof of decomposition formula in \cite{LP}. \\

\noindent \textit{Step 1}. By \cite[Theorem 7.4.5]{LF}, $H^q(\mathscr{M}_{T})=0$ for $q>2r$. For any $n$-torsion finite $\pi_T$-module $A$ and $i \geq 3$, 
$$
H^i(\pi_T, A) \cong \prod_{v \in T}H^i(\pi_v, A) 
$$
by \cite[8.6.10(ii)]{NSW}. Combining these two facts with the definition of cohomology with compact support, we obtain that the top row of the following diagram is exact (cf. \cite[Section 3 (1)]{LP}). 
\begin{equation}\label{diag1}
\begin{tikzcd}
H^{2r+2}(\pi_T, \mathscr{M}_{T}) \arrow[d, "j^{2r+2}"] \arrow[r, "i_T^*"] & H_T^{2r+2} \arrow[d, equal] \arrow[r] & H^{2r+3}_c(\pi_T, \mathscr{M}_{T}) \arrow[r] \arrow[d, "j_c^{2r+3}"] & 0 \\
H_c^{2r+2}(X_T, \Z/n\Z) \arrow[r] & P_c^{2r+2}(X_T, \Z/n\Z) \arrow[r] & \widehat{H}_c^{2r+3}(X_T, \Z/n\Z) \arrow[r] & 0
\end{tikzcd}
\end{equation}
Note that the map $i_T^*$ is given by $H^i(\pi_T, \mathscr{M}_T) \rightarrow H_T^i$ ($\rho \mapsto (\rho \circ i_v)_{v \in T}$) and the composition
$$
P_c^{2r+2}(X_T, \Z/n\Z) \rightarrow \widehat{H}_c^{2r+3}(X_T, \Z/n\Z) \xrightarrow[\simeq]{\inv_T} (\mu_{n, X_T}(X_T)^D)^{\otimes (r+1)}
$$
is identified with the map
$$
P_c^{2r+2}(X_T, \Z/n\Z) \rightarrow H^0(X_T, \Z/n\Z^{\vee}(r+1))^D
$$
given by Theorem \ref{thm223} (see Section \ref{Sub33}). \\

\noindent \textit{Step 2}. Now define the map $\inv'_T$ by
$$
\inv'_T : H^{2r+3}_c(\pi_T, \mathscr{M}_{T}) \xrightarrow{j^{2r+3}_c} \widehat{H}^{2r+3}_c(X_T, \Z/n \Z) \xrightarrow[\simeq]{\inv_T} (\mu_{n, X_T}(X_T)^D)^{\otimes (r+1)} \xrightarrow{\simeq} \frac{1}{n}\Z/\Z.
$$
Then the commutativity of the diagram below follows from the commutativity of the right square of the diagram (\ref{diag1}).
\[
\begin{tikzcd}
H_T^{2r+2} \arrow[d] \arrow[r] & H^{2r+3}_c(\pi_T, \mathscr{M}_{T}) \arrow[d, "\inv'_T"] \\
P_c^{2r+2}(X_T, \Z/n\Z) \arrow[r, "\sum"] & \frac{1}{n}\Z/\Z
\end{tikzcd}
\]
Note that if $\mathscr{M}_{T}$ is a complex of locally constant sheaves, then $j_c^{2r+3}$ (so $\inv'_T$) is an isomorphism by \cite[Proposition 2.2.9]{ADT}. \\

\noindent \textit{Step 3}. We can summarize the results of Step 1, 2 and the maps between cohomology groups constructed in Section \ref{Sub31} by the following commutative diagram.
\[
\begin{tikzcd}[column sep=1.5em]
& {\widehat{H}^{2r+3}_c(X, \Z/n \Z)} \arrow[r]
& {\widehat{H}^{2r+3}_c(X_{T_0}, \Z/n \Z)} \arrow[d, "\simeq"] \arrow[r, "\inv_{T_0}", "\simeq"'] 
& (\mu_{n, X_{T_0}}(X_{T_0})^D)^{\otimes (r+1)} \arrow[d, "\simeq"] \\
{H^{2r+3}(\pi^{\un}, \mathscr{M})} \arrow[d] \arrow[r] \arrow[ru, "j^{2r+3}_{\un}"] \arrow[rrrd, bend left=17, "CS"]
& {H^{2r+3}_c(\pi, \mathscr{M}_{S_{\infty}})} \arrow[d] \arrow[u, "j_c^{2r+3}"'] 
& {\widehat{H}^{2r+3}_c(X_T, \Z/n \Z)} \arrow[r, "\inv_T", "\simeq"'] 
& (\mu_{n, X_T}(X_T)^D)^{\otimes (r+1)} \arrow[d, "\simeq"] \\
{H^{2r+3}_c(\pi_{T_f}, \mathscr{M}_{T_f})} \arrow[r] 
& {H^{2r+3}_c(\pi_T, \mathscr{M}_{T})} \arrow[ru, "j^{2r+3}_c"'] \arrow[rr, "\inv'_T"] 
& 
& \frac{1}{n}\Z/\Z
\end{tikzcd}
\]

\vspace{3mm}

\noindent \textit{Step 4}. Let $\rho = [\omega] \in H^{2r+3}(\pi^{\un}, \mathscr{M})$. By the vanishing of $H^{2r+3}(\pi_{T_f}, \mathscr{M}_{T_f})$, there exists 
$$
b'_+ \in C^{2r+2}(\pi_{T_f}, \mathscr{M}_{T_f})
$$
such that $\kappa_{T_f}^*(\omega) = db'_+$. Denote 
\begin{center}
$b_+  \in C^{2r+2}(\pi_{T}, \mathscr{M}_{T})$ and $b_{+,v} \in C^{2r+2}(\pi_v, \mathscr{M}_{T, v})$ ($v \in T_f$)
\end{center}
where $b_+$ is induced by $p_T$ and $\mathscr{M}_{T_f} \subset \mathscr{M}_{T}$, $b_{+,v}$ is induced by $i_{v, T_f}$ and $\mathscr{M}_{T_f} \subset \mathscr{M}_{T} \rightarrow \mathscr{M}_{T, v}$. Following the argument of \cite[Section 3 (2), (6)]{LP}, we get the following results. \\
$\bullet$ There exists $\widetilde{b_{-,v}} \in C^{2r+2}(\pi_v/I_v, \mathscr{N}_{T, v})$ such that $\kappa_v^*(\omega) \mid_{\pi_v/I_v} = d \widetilde{b_{-,v}}$. \\
$\bullet$ Let $b_{-,v} \in C^{2r+2}(\pi_v, \mathscr{M}_{T, v})$ be the image of $\widetilde{b_{-,v}}$ under the map given by the projection $\pi_v \rightarrow \pi_v/I_v$ and the inclusion $\mathscr{N}_{T, v} \subset \mathscr{M}_{T, v}$. Then the image of $\rho$ in $H^{2r+3}_c(\pi_T, \mathscr{M}_T)$ is given by $[(\kappa_T^*(\omega), (b_{-,v})_{v \in T})]$. (See \cite[Section 2.2]{LP} for notations.) \\

\noindent \textit{Step 5}. Now $db_+ = d(b'_+ \circ p_T) = \kappa_T^*(\omega)$ so
$$
[(\kappa_T^*(\omega), (b_{-,v})_{v \in T})] = [(db_+, (b_{-,v})_{v \in T})] = [(0, (b_{-,v}-b_{+,v})_{v \in T})]
$$
in $H^{2r+3}_c(\pi_T, \mathscr{M}_T)$. Since $b_{+,v}=b_{-,v}=0$ for every $v \in S_{\infty}$, from the diagram in Step 3 we obtain
\begin{equation*}
\begin{split}
CS(\rho) & = \inv'_T([(\kappa_T^*(\omega), (b_{-,v})_{v \in T})])  \\
& = \inv'_T([(0, (b_{-,v}-b_{+,v})_{v \in T})]) \\
& = \sum_{v \in T_f} \inv_v([b_{-,v}-b_{+,v}]). 
\end{split}
\end{equation*}
(Last equality is from the diagram in Step 2 and the relation $\sum = \sum_{v \in T} \inv_v$.)
\begin{theorem}
Let $F$ be a number field, $n \geq 2$ be an integer and $X$ be a regular, flat, separated scheme of finite type of relative dimension $r$ over $B = \Spec \mathcal{O}_F$ such that $\mu_{n, X}(X) \cong \Z/n\Z$. For any finite set $T_f$ of finite primes of $F$ containing all primes dividing $n$ with the notations above, we have the decomposition formula
$$
CS(\rho) := (\inv_{T_0} \circ \beta^{\text{\'et}}_{S_{\infty}, T_0} \circ j^{2r+3}_{\un})(\rho) = \sum_{v \in T_f} \inv_v([b_{-,v}-b_{+,v}]). 
$$
\end{theorem}


\section*{Acknowledgments}

The author would like to thank Jeehoon Park for his encouragement. The author also thanks Alexander Schmidt for answering my questions. 


\vspace{3mm}

\footnotesize{
\textsc{Jungin Lee: Department of Mathematics, Pohang University of Science and Technology, 77 Cheongam-ro, Nam-gu, Pohang, Gyeongbuk, Republic of Korea 37673.} 

\textit{E-mail address}: \changeurlcolor{black}\href{mailto:moleculesum@postech.ac.kr}{moleculesum@postech.ac.kr}

\end{document}